\documentclass[12pt]{article}

\usepackage{amsmath,amssymb,amsthm,amsfonts,latexsym,hyperref,enumerate}
\usepackage{xcolor}
\usepackage{graphicx}
\usepackage{enumitem}

\newtheorem{lemma}{Lemma}[section]
\newtheorem{theorem}[lemma]{Theorem}
\newtheorem{example}[lemma]{Example}

\newtheorem{definition}[lemma]{Definition}
\newtheorem{remark}[lemma]{Remark}

\begin{document}

\title{Constructing cospectral hypergraphs}

\author{
Aida Abiad
\thanks{\texttt{a.abiad.monge@tue.nl}, Department of Mathematics and Computer Science, Eindhoven University of Technology, The Netherlands}
\thanks{Department of Mathematics: Analysis, Logic and Discrete Mathematics, Ghent University, Belgium} 
\thanks{Department of Mathematics and Data Science, Vrije Universiteit Brussel, Belgium} 
\and
Antonina P. Khramova
\thanks{\texttt{a.khramova@tue.nl}, Department of Mathematics and Computer Science, Eindhoven University of Technology, The Netherlands} 
}

\date{}
\maketitle

\begin{abstract}
Spectral hypergraph theory mainly concerns using hypergraph spectra to obtain structural information about the given hypergraphs. The study of cospectral hypergraphs is important since it reveals which hypergraph properties cannot be deduced from their spectra. In this paper, we show a new method for constructing cospectral uniform hypergraphs using two well-known hypergraph representations: adjacency tensors and adjacency matrices.
\end{abstract}




\section{Introduction}

Spectral hypergraph theory seeks to deduce structural properties about the hypergraph using its spectra. This field has received a lot of attention in the last two decades, see for example \cite{C2012,CD2015,FL1996,LZ2017,R2002,SSP2022,S2013,SQH2015,W2022,XW2019,ZKSB2017}. The study of cospectral hypergraphs is important since it reveals
which hypergraph properties cannot be deduced from their spectra. While the construction of cospectral graphs has
been investigated extensively in the literature (see e.g. \cite{ah2012, gm1982,HH1988,QJW2020,s1973}), much less is known about the construction of cospectral hypergraphs. In this regard, Bu, Zhou, and Wei \cite{bzw2014} presented a switching method for contructing $E$-cospectral hypergraphs which is based on the widely used Godsil-McKay switching (GM-switching) for graphs \cite{gm1982}. Another extension of GM-switching was shown by Banerjee and Sarkar in \cite{SB2020} for a matrix representation of a hypergraph using a natural generalization of the adjacency matrix for simple graphs. 

In this paper, we show a new method for constructing $E$-cospectral hypergraphs (with respect to their $E$-characteristic polynomials). We also propose a new method for constructing uniform hypergraphs which are cospectral with respect to their adjacency matrices. Both of these methods are  based on the recently introduced graph switching by Wang, Qiu and Hu (WQH-switching) \cite{WQH2019, QJW2020}.

\section{Cospectral and $E$-cospectral hypergraphs using adjacency tensors}

For a positive integer $n$, let $[n] = \{1, \ldots , n\}$. An order $k$ dimension $n$ \emph{tensor} $\mathcal{A}=(a_{i_1\cdots i_k}) \in \mathbb{C}^{n\times \cdots \times n}$ is a multidimensional array with $n^k$ entries, where $i_j \in [n]$, $j=1,\ldots, k$. For example, in case $k=1$, $\mathcal{A}$ is a column vector of dimension $n$, and in case $k=2$, $\mathcal{A}$ is an $n\times n$ matrix.

The following tensor multiplication was introduced by Shao \cite{S2013} as a generalization of the matrix multiplication.

\begin{definition}\cite{S2013}\label{def:tensorproduct}
Let $\mathcal{A}$ and $\mathcal{B}$ be order $m\geq 2$ and order $k\geq 1$, dimension $n$ tensors, respectively. The product $\mathcal{A}\mathcal{B}$ is the following tensor $\mathcal{C}$ of order $(m-1)(k-1) +1$ and dimension $n$ with entries:
$$c_{i\,\alpha_1 \cdots  \alpha_{m-1}}=\sum_{i_2,\ldots, i_m\in [n]} a_{i\, i_2 \cdots  i_{m}} b_{i_2\,\alpha_{1}} \cdots b_{i_m\,\alpha_{m-1}},$$
where $i\in [n]$, $\alpha_1,\ldots,\alpha_{m-1} \in [n]^{k-1}$.
\end{definition}

In particular, according to \cite[Example 1.1]{S2013}, for an order $k\geq 2$ dimension $n$ tensor $\mathcal{A}$ and a vector $x = (x_1,\ldots, x_n)^\top$ we can derive that the product $\mathcal{A}x$ is a vector with $i$-th component calculated by
$$(\mathcal{A}x)_i=\sum_{i_2,\ldots, i_k\in [n]} a_{i\, i_2 \cdots  i_{k}} x_{i_2}\cdots x_{i_k}.$$

In 2005, Qi \cite{Q2005} and Lim \cite{L2005} independently introduced the concept of tensor eigenvalues with two different definitions. As we will see below, both of them generalize the notion of matrix eigenvalue in their own way. Since then, a vast number of authors have used such definitions to study spectral properties of hypergraphs \cite{C2012,CD2015,S2013, SQH2015,W2022,XW2019,ZKSB2017}. Next, we introduce the  definitions of characteristic and $E$-characteristic polynomials of a tensor.  

Let $\mathcal{A}$ be an order $k$ dimension $n$ tensor. A number $\lambda \in \mathbb{C}$ is called an \emph{eigenvalue} of $\mathcal{A}$ if there exists a nonzero vector $x \in \mathbb{C}^n$ such that $\mathcal{A}x = \lambda x^{[k-1]}$, where $x^{[k-1]}=(x_1^{k-1},\ldots, x_n^{k-1})^\top$. The \emph{determinant} of $\mathcal{A}$, denoted by $\text{det}(\mathcal{A})$, is the resultant of the system of polynomials $f_i(x_1,\ldots, x_n)=(\mathcal{A}x)_i$ for all $i\in[n]$. The \emph{characteristic polynomial} of $\mathcal{A}$ is defined as $\Phi_{\mathcal{A}}(\lambda)=\det(\lambda \mathcal{I}_n - \mathcal{A})$, where $\mathcal{I}_n$ is the \emph{unit tensor} of order $k$ and dimension $n$, \emph{i.e.} a tensor of elements $\delta_{i_1\cdots i_k}$ such that
$$\delta_{i_1\cdots i_k}=\begin{cases}
1, & i_1=i_2=\cdots =i_k,\\
0, & \text{otherwise}.
\end{cases}$$
It is known that the eigenvalues of $\mathcal{A}$ are exactly the roots of $\Phi_{\mathcal{A}}(\lambda)$ \cite{S2013}.

On the other hand, for an order $k\geq 2$ dimension $n$ tensor $\mathcal{A}$, a number $\lambda\in \mathbb{C}$ is called an $E$-\emph{eigenvalue} of $\mathcal{A}$ if there exists a nonzero vector $x\in\mathbb{C}^n$ such that $\mathcal{A}x = \lambda x$ and $x^\top x=1$. In \cite{Q2007},
the $E$-\emph{characteristic polynomial} of $\mathcal{A}$ is defined as

$$\phi_{\mathcal{A}}(\lambda)=\begin{cases}
\operatorname{Res}_x (\mathcal{A}x-\lambda(x^\top x)^{\frac{k-2}2}x), & k \text{ is even}, \\
\operatorname{Res}_{x,\beta} \left(\begin{smallmatrix}
\mathcal{A}x-\lambda\beta^{k-2}x \\
x^\top x -\beta^2
\end{smallmatrix}\right), & k \text{ is odd}, \\
\end{cases}$$
where ‘Res’ is the resultant of the system of polynomials. It is known that $E$-eigenvalues
of $\mathcal{A}$ are roots of $\phi_{\mathcal{A}}(\lambda)$ \cite{Q2007}. If $k=2$, then $\phi_{\mathcal{A}}(\lambda)=\Phi_{\mathcal{A}}(\lambda)$ is just the characteristic
polynomial of the square matrix $\mathcal{A}$.

A hypergraph $G=(V(G),E(G))$ is called $k$-\emph{uniform} if each edge of $G$ contains exactly $k$ distinct vertices. All hypergraphs in this paper are uniform and simple. The \emph{adjacency tensor} of $G$, denoted by $\mathcal{A}_G$, is an order $k$ dimension $|V(G)|$ tensor with entries \cite{C2012}:

$$a_{i_1\, i_2 \cdots  i_{k}}=\begin{cases}
\frac1{(k-1)!}, & \{i_1,i_2,\dots,i_k\}\in E(G), \\
0, & \text{otherwise.}
\end{cases}$$

We say that two $k$-uniform hypergraphs are \emph{cospectral} (\emph{$E$-cospectral}) if their adjacency tensors have the same characteristic polynomial (E-characteristic polynomial).

In order to state our main result in this section, we need some preliminary work.
We shall follow the same notation as in \cite{bzw2014}.

Let $G=(V(G), E(G))$ be a $k$-uniform hypergraph. The \emph{degree} of a vertex $u\in V(G)$ is the number of edges that contain $u$.
For any edge $\{u_1, \dots, u_k\} \in E(G)$, we say that $u_1$ is a \emph{neighbour} of $\{u_2,\ldots, u_k\}$. The \emph{neighbourhood} of $\{u_2,\dots,u_k\}$ is denoted by $\Gamma(u_2,\dots,u_k)$.

We say that $\mathcal{A}$ is a \emph{symmetric tensor} if $a_{i_1\,i_2\cdots i_k} = a_{i_{\sigma{(1)}}\,i_{\sigma{(2)}}\cdots i_{\sigma{(k)}}}$ for any permutation $\sigma$ on $[k]$. From the definition of the adjacency tensor, it is easy to observe that for a hypergraph $G$, the tensor $\mathcal{A}_G$ is symmetric.

The following lemma can be obtained from \cite[Eq. (2.1)]{S2013}.

\begin{lemma}\label{l.count} Let $\mathcal{A}=(a_{i_1\cdots i_k})$ be an order $k\geq2$ dimension $n$ tensor, and let $Q=(q_{i\,j})$ be an $n\times n$ matrix. Then
$$(Q\mathcal{A}Q^\top)_{i_1\cdots i_k} = \sum\limits_{j_1,\dots,j_k\in[n]} a_{j_1\cdots j_k}q_{i_1\,j_1}q_{i_2\,j_2}\cdots q_{i_k\,j_k}.$$
\end{lemma}

From Lemma~\ref{l.count} we obtain the following result.

\begin{lemma}\label{l.sym}
Let $\mathcal{A}'=Q\mathcal{A}Q^\top$, where $\mathcal{A}$ is a tensor of dimension $n$ and $Q$ is an $n\times n$ matrix. If $\mathcal{A}$ is symmetric, then $\mathcal{A}'$ is symmetric.
\end{lemma}

Additionally, let $Q$ be a real orthogonal matrix. In \cite{S2013}, Shao pointed out that $\mathcal{A}$ and $\mathcal{A}'=Q\mathcal{A}Q^\top$ are orthogonally similar tensors as defined by Qi \cite{Q2005}, which implies that they have the same set of $E$-eigenvalues. In this case, the $E$-characteristic polynomials are also the same (see also \cite{LQZ2013}):

\begin{lemma}\label{l.ortho} Let $\mathcal{A}'=Q\mathcal{A}Q^\top$, where $\mathcal{A}$ is a tensor of dimension $n$ and $Q$ is an $n\times n$ real orthogonal matrix. Then $\mathcal{A}$ and $\mathcal{A}'$ have the same E-characteristic polynomial.
\end{lemma}

Let $\mathcal{I}_n$ be the \emph{unit tensor} of order $k$ and dimension $n$, \emph{i.e.} a tensor of elements $\delta_{i_1\cdots i_k}$ such that
$$\delta_{i_1\cdots i_k}=\begin{cases}
1, & i_1=i_2=\cdots =i_k,\\
0, & \text{otherwise}.
\end{cases}$$
A claim analogous to Lemma~\ref{l.ortho} can be made for characteristic polynomials of tensors, as it is a straightforward consequence of~\cite[Theorem 2.1]{S2013}:
\begin{lemma}\label{l.ortho2}
Let $\mathcal{A}'=Q\mathcal{A}Q^\top$, where $\mathcal{A}$ is a tensor of dimension $n$ and $Q$ is an $n\times n$ real orthogonal matrix such that $Q\mathcal{I}_nQ^\top=\mathcal{I}_n$. Then $\mathcal{A}$ and $\mathcal{A}'$ have the same characteristic polynomial.
\end{lemma}
Note that the additional identity $Q\mathcal{I}_nQ^\top=\mathcal{I}_n$ does not hold in general for a real orthogonal matrix $Q$, making cospectrality of tensors a stronger property than $E$-cospectrality.

Lemma~\ref{l.ortho} will be the key ingredient for proving the $E$-cospectrality of the hypergraphs constructed using the method described in Section \ref{sec:Ecospectralityswitching}.

\subsection{Constructing $E$-cospectral hypergraphs using adjacency tensors}\label{sec:Ecospectralityswitching}

Inspired by the WQH-switching~\cite{WQH2019} for graphs, we propose a method to construct $E$-cospectral hypergraphs.

\begin{theorem}\label{t.main}  
Let $G$ be a $k$-uniform hypergraph on $n$ vertices that satisfies the following conditions:
\begin{enumerate}
\item The vertex set $V(G)$ is partitioned into three sets $C_1\cup C_2\cup D$ with $|C_1|=|C_2|=t$.
\item For any edge $\{ u_1,\dots,u_k\}\in E(G)$, there is at most one vertex in $C_1\cup C_2$, {\it i.e.} $|\{ u_1,\dots,u_k\}\cap (C_1\cup C_2)|\leq 1$.
\item For any $k-1$ distinct vertices $u_2,\dots,u_k$ from $D$, we have $\Gamma(u_2,\dots,u_k)\cap (C_1\cup C_2)\in\{C_1, C_2\}$ or $|\Gamma(u_2,\dots,u_k)\cap C_1|=|\Gamma(u_2,\dots,u_k)\cap C_2|$.
\end{enumerate}

To construct a hypergraph $H$, for any $\{u_2,\dots,u_k\}\subseteq D$ that has all its neighbours in $C_1$ (or $C_2$), switch the adjacency of $\{u_1,\dots,u_k\}$ for all $u_1\in C_1\cup C_2$. Then $H$ is a $k$-uniform $E$-cospectral hypergraph with $G$.
\end{theorem}

\begin{proof}
To prove this result, we will show that $$\mathcal{A}_H=Q \mathcal{A}_G Q^\top,$$ where $\mathcal{A}_G$ and $\mathcal{A}_H$ are the adjacency tensors of $G$ and $H$, and
$$Q=\left(\begin{matrix}
I_t-\frac1t J_t & \frac1t J_t & 0 \\
\frac1t J_t & I_t-\frac1t J_t & 0 \\
0 & 0 & I_{n-2t}
\end{matrix}\right).$$
Since $Q$ is a real orthogonal matrix, the $E$-cospectrality of $G$ and $H$ follows from Lemma~\ref{l.ortho}.

Let $\mathcal{A}'=Q \mathcal{A}_G Q^\top$. According to Lemma~\ref{l.sym}, $\mathcal{A}'$ is symmetric. By Lemma~\ref{l.count}, $$(\mathcal{A}')_{i_1\cdots i_k}=\sum\limits_{j_1,\dots,j_k\in V(G)} a_{j_1\cdots j_k}q_{i_1\,j_1}\cdots q_{i_k\,j_k}.$$
We need to show that $\mathcal{A}'=\mathcal{A}_H$.

First, consider the case $|\{i_1,\dots,i_k\}\cap (C_1\cup C_2)|=0$, or $\{i_1,\dots,i_k\}\subseteq D$. Then, for any $s\in[k]$, $q_{i_s\,j_s}=1$ if and only if $i_s=j_s$ and is equal to $0$ otherwise. Hence
$$(\mathcal{A}')_{i_1\cdots i_k}=a_{i_1\cdots i_k}\text{ if }\{i_1,\dots,i_k\}\subseteq D.$$

Next, if $|\{i_1,\dots,i_k\}\cap (C_1\cup C_2)|\geq2$, then $a_{i_1\cdots i_k}=0$ since every edge of $H$ has no more than one vertex in $C_1\cup C_2$. We consider three cases for $\{j_1,\dots,j_k\}$:
\begin{itemize}
\item If $|\{j_1,\dots,j_k\}\cap (C_1\cup C_2)|=0$ then without loss of generality we may assume that $i_1\in (C_1\cup C_2)$. Since $j_1\in D$, we have $q_{i_1\,j_1}=0$.
\item If $|\{j_1,\dots,j_k\}\cap (C_1\cup C_2)|\geq2$ then $\{j_1,\dots,j_k\}$ is not an edge of $G$ and $a_{j_1\cdots j_k}=0$.
\item If $|\{j_1,\dots,j_k\}\cap (C_1\cup C_2)|=1$ then without loss of generality we may assume that $i_1, i_2 \in (C_1 \cup C_2)$. Then, either $j_1 \in D$ or $j_2 \in D$, hence we have $q_{i_1\,j_1} = 0$ or $q_{i_2\,j_2} = 0$.
\end{itemize}
This argument implies that every term in the sum that defines $(\mathcal{A}')_{i_1\cdots i_k}$ is equal to zero, and so
$$(\mathcal{A}')_{i_1\cdots i_k}=0=a_{i_1\cdots i_k}\text{ if }|\{i_1,\dots,i_k\}\cap (C_1\cup C_2)|\geq2.$$

The final case is $|\{i_1,\dots,i_k\}\cap (C_1\cup C_2)|=1$. Without loss of generality we may assume $i_1\in (C_1\cup C_2)$ and $i_2,\dots,i_k\in D$. Since $q_{i_s\,j_s}=0$ for any $s\in\{2,\dots,k\}$ unless $j_s=i_s$ when $q_{i_s\,j_s}=1$, and also $q_{i_1\,j_1}=0$ if $j_1\in D$, we have
$$(\mathcal{A}')_{i_1\cdots i_k}=\sum\limits_{j_1\in (C_1\cup C_2)} a_{j_1\,i_2\cdots i_k}q_{i_1\,j_1}.$$
We assume $i_1\in C_1$, and the case $i_1\in C_2$ can be considered analogously.

There are three possibilities:
\begin{description}
\item[Case 1.] The set $\{i_2,\dots,i_k\}$ has every vertex of $C_1$ as a neighbour and no neighbours in $C_2$. This implies $i_1$ is one of these neighbours, so $a_{i_1\,i_2\cdots i_k}=\frac1{(k-1)!}$. Then 
\begin{align*}
(\mathcal{A}')_{i_1\cdots i_k}&=\sum\limits_{j_1\in (C_1\cup C_2)} a_{j_1\,i_2\cdots i_k}q_{i_1\,j_1}= \frac1{(k-1)!}\sum\limits_{j_1\in C_1}q_{i_1\,j_1}\\
&=\frac1{(k-1)!}\sum\limits_{j_1\in C_1, j_1\neq i_1} q_{i_1\,j_1} + \frac1{(k-1)!}q_{i_1\,i_1}\\
&=\frac1{(k-1)!}\left(- (t-1)\frac1t + \frac{t-1}t\right)=0\neq a_{i_1\,i_2\cdots i_k}.
\end{align*}

\item[Case 2.] The set $\{i_2,\dots,i_k\}$ has every vertex of $C_2$ as a neighbour and no neighbours in $C_1$, implying $a_{i_1\,i_2\cdots i_k}=0$. Then 
\begin{align*}
(\mathcal{A}')_{i_1\cdots i_k}&=\sum\limits_{j_1\in (C_1\cup C_2)} a_{j_1\,i_2\cdots i_k}q_{i_1\,j_1}=\sum\limits_{j_1 \in C_2} a_{j_1\,i_2\cdots i_k}q_{i_1\,j_1}\\
&= t\cdot\frac1{(k-1)!}\cdot\frac1t
=\frac1{(k-1)!}\neq a_{i_1\,i_2\cdots i_k}.    
\end{align*}
 
\item[Case 3.] The set $\{i_2,\dots,i_k\}$ has $r\leq t$ neighbours in $C_1$ as well as in $C_2$. Then
\begin{align*}
(\mathcal{A}')_{i_1\cdots i_k}&=\sum\limits_{j_1 \in C_1} a_{j_1\,i_2\cdots i_k}q_{i_1\,j_1}+\sum\limits_{j_1 \in C_2} a_{j_1\,i_2\cdots i_k}q_{i_1\,j_1}\\
&=\sum\limits_{j_1 \in C_1} a_{j_1\,i_2\cdots i_k}q_{i_1\,j_1}+r\cdot\frac1{(k-1)!}\cdot\frac1t.
\end{align*}
Depending on whether $i_1$ is one of the $r$ neighbours in $C_1$ or not, we have either
$$(\mathcal{A}')_{i_1\cdots i_k}=-\frac1t\cdot r\cdot\frac1{(k-1)!}+r\cdot\frac1{(k-1)!}\cdot\frac1t=0\text{ when }a_{i_1\,i_2\cdots i_k}=0,$$
or
\begin{align*}
 (\mathcal{A}')_{i_1\cdots i_k}&=-\frac1t\cdot (r-1)\cdot\frac1{(k-1)!}+\frac{t-1}t \cdot\frac1{(k-1)!}+r\cdot\frac1{(k-1)!}\cdot\frac1t\\
 &=\frac1{(k-1)!} \text{    when }a_{i_1\,i_2\cdots i_k}\neq 0.
\end{align*}
Either way we derive that in Case 3,
$$(\mathcal{A}')_{i_1\cdots i_k}=a_{i_1\,i_2\cdots i_k}.$$   
\end{description}

By considering every possible case we see that $\mathcal{A}'$ is the adjacency tensor of a hypergraph that can be obtained from $G$ by switching adjacency of every edge $\{i_1,i_2,\dots,i_k\}$ such that $\{i_2,\dots,i_k\}\subseteq D$ and $i_1\in (C_1\cup C_2)$ in case $\Gamma(i_2,\dots,i_k)\cap(C_1\cup C_2)\in\{C_1,C_2\}$, which is exactly $H$ by construction. Hence $\mathcal{A}'=\mathcal{A}_H$ as required.
\end{proof}

\begin{remark}Note that the matrix $Q$ defined in the proof of Theorem~\ref{t.main} in general does not satisfy $Q\mathcal{I}_nQ^\top=\mathcal{I}_n$. Hence the hypergraph~$H$ obtained from $G$ as the result of switching described in Theorem~\ref{t.main} is $E$-cospectral but not necessarily cospectral with~$G$.
\end{remark}

We will refer to the switching of Theorem~\ref{t.main} as \emph{$E$-WQH switching}.
The paper~\cite{bzw2014} describes a switching to construct $E$-cospectral hypergraphs by a method which is similar to GM-switching~\cite{gm1982}. This will be referred to as \emph{$E$-GM switching}.

Note that the switching described in Theorem~\ref{t.main} is based on the simplified version of WQH-switching for ordinary graphs. However, the cospectrality of the hypergraphs constructed using the more general version of the switching can also be argued in a similar manner. Indeed, let $G$ be a $k$-uniform hypergraph such that $V(G)$ admits partition $C_1\cup C_2\cup \dots \cup C_{2m}\cup D$ for some integer $m\geq 1$ with $|C_i|=|C_{i+1}|$ for any odd $i\in[2m]$, any edge has at most one vertex in $V(G)\setminus D$, and for each odd $i\in[2m]$ and a subset $\{u_2,\dots,u_k\}\subseteq D$ we have either $\Gamma(u_2,\dots,u_k)\cap(C_i\cup C_{i+1})\in\{C_i,C_{i+1}\}$ or $|\Gamma(u_2,\dots,u_k)\cap C_i|=|\Gamma(u_2,\dots,u_k)\cap C_{i+1}|$. Then a hypergraph $H$ is constructed by taking all subsets $\{u_2,\dots,u_k\}\subseteq D$ such that $\Gamma(u_2,\dots,u_k)\cap(C_i\cup C_{i+1})\in\{C_i,C_{i+1}\}$ for some odd $i\in[2m]$ and switching adjacency between the subset and all the vertices in $C_i\cup C_{i+1}$.
A similar observation can be made regarding to $E$-GM switching, which is based on the simplified version of GM-switching but can be extended to its general version~\cite{gm1982}.

\begin{remark}\label{sec:remarksssize2}
If $G$ is a hypergraph that admits the conditions of Theorem~\ref{t.main} and $|C_1\cup C_2|=2$, then the hypergraph constructed as a result of $E$-WQH switching is isomorphic to $G$. It is easily observed that the isomorphism is the permutation of the two vertices of $C_1\cup C_2$. The same observation is true for $E$-GM switching.
\end{remark}

In $E$-GM switching, it is required to partition the vertices of a hypergraph $G$ into two sets $C$ and $D$, where no two vertices of $C$ are in the same edge, and any $(k-1)$-subset of $D$ has either $0$, $\frac{|C|}2$, or $|C|$ neighbours in $C$. We will call $C$ the \emph{switching set} of $G$.

\begin{remark} If $G$ is a $k$-uniform hypergraph satisfying the conditions of Theorem~\ref{t.main} and $|C_1\cup C_2|=4$, then the hypergraph constructed as a result of $E$-WQH switching is isomorphic to the one constructed using $E$-GM switching. Indeed, any subset of $k-1$ vertices in $D$ must have $0$, $2$, or $4$ neighbours in $C_1\cup C_2$. This implies that the conditions of $E$-GM switching are satisfied for $G$ with the switching set $C:=C_1\cup C_2$. Applying $E$-GM switching is equivalent to applying $E$-WQH switching and a permutation of the two vertices in $C_1$ and the two vertices in $C_2$.
\end{remark}

Finally we provide an example of two $E$-cospectral non-isomorphic hypergraphs that can be obtained through the new $E$-WQH switching from Theorem~\ref{t.main} but not via the $E$-GM switching shown in \cite{bzw2014}.
\begin{example}
Let $G$ be a $3$-uniform hypergraph on $9$ vertices $u_1,u_2,u_3,u_4,u_5,u_6,v_1,v_2,v_3$ with edges
\begin{gather*}
    \{v_1,v_2,u_1\},\;\{v_1,v_2,u_2\},\;\{v_1,v_2,u_3\},\;\{v_2,v_3,u_1\},\;\{v_2,v_3,u_4\},\\
    \{v_1,v_3,u_2\},\;\{v_1,v_3,u_3\},\;\{v_1,v_3,u_4\},\;\{v_1,v_3,u_5\}.
\end{gather*}
If we take $C_1:=\{u_1,u_2,u_3\}$ and $C_2:=\{u_4,u_5,u_6\}$, then every edge has exactly one vertex in $C_1\cup C_2$, and every $2$-subset of $\{v_1,v_2,v_3\}$ has either three neighbours in $C_1$ and none in $C_2$, or the same number of neighbours in both $C_1$ and $C_2$. Hence this hypergraph admits the switching described in Theorem~\ref{t.main}, and the result is a hypergraph $H$ with edge set
\begin{gather*}
    \{v_1,v_2,u_4\},\;\{v_1,v_2,u_5\},\;\{v_1,v_2,u_6\},\;\{v_2,v_3,u_1\},\;\{v_2,v_3,u_4\},\\
    \{v_1,v_3,u_2\},\;\{v_1,v_3,u_3\},\;\{v_1,v_3,u_4\},\;\{v_1,v_3,u_5\},
\end{gather*}
\emph{i.e.} only the first three edges are switched.
$G$ and $H$ are clearly $E$-cospectral but not isomorphic, since only one of them has an isolated vertex. Moreover, the hypergraph $H$ cannot be constructed from $G$ using the $E$-GM switching, since the switching set would have to include $C_1\cup C_2$ and simultaneously cannot include any of the vertices from $\{v_1,v_2,v_3\}$. The set $C_1\cup C_2$ by itself does not satisfy the conditions of $E$-GM switching.
\end{example}

\section{Constructing cospectral hypergraphs using matrix representation}\label{sec:matrixcospectralityswitching}

One of the disadvantages of studying the hypergraph spectrum using tensors, as in the previous section, is the computational complexity: computing eigenvalues of the adjacency tensor is known to be an NP-hard problem \cite{HL2013}. On the other hand, a different way of representing a hypergraph on $n$ vertices can be found in the literature (see, for example, \cite{FL1996,LZ2017,SSP2022,R2002}). Here we have an $n\times n$ matrix where its $(i,j)$-entry is the number of edges that vertices labelled $i$ and $j$ share. We use the slightly altered definition from \cite{B2021} for the \emph{adjacency matrix} $A=(A_{i\,j})$ of a $k$-uniform hypergraph $G$:

$$A_{i\,j}=\begin{cases}
0, & i=j, \\
\frac{|\{e\in E(G) | i,j\in e \}|}{k-1}, & i\neq j.
\end{cases}$$

Two hypergraphs are said to be \emph{cospectral} with respect to their matrix representation if their adjacency matrices have the same spectrum.

A first switching method to construct cospectral hypergraphs regarding the previous adjacency matrix was shown by Sarkar and Banerjee in~\cite{SB2020}. Their method is analogous to GM-switching for ordinary graphs in its most general form. In this section we show a different switching method which extends WQH-switching~ \cite{WQH2019, QJW2020} to hypergraphs.

\begin{theorem}\label{t.matrix} Let $G=(V(G),E(G))$ be a $k$-uniform hypergraph whose vertex set admits a partition $C_1\cup C_2\cup \dots\cup C_{2m}\cup D$ for some integer $m\geq 1$, and such that the following conditions for $G$ hold:
\begin{enumerate}
\item $|C_i|=t$ for all $i\in[2m]$ and some integer $t$, while $|D|=k-1$.
\item Any edge of $G$ has either $0$ or $k-1$ vertices in $D$.
\item For any odd $i\in[2m]$, we have either $\Gamma(D)\cap(C_i\cup C_{i+1})=C_i$ or $|\Gamma(D)\cap C_i|=|\Gamma(D)\cap C_{i+1}|$, where $\Gamma(D)$ denotes the set of neighbours of $D$, or a subset of vertices $v$ such that $\{v,D\}\in E(G)$.
\item For the adjacency matrix $A$ and each $i,j\in[2m]$ there exists an integer $\alpha_{i\,j}$ such that
$$\sum\limits_{u\in C_i} A_{u\,v}=\sum\limits_{u\in C_i} A_{v\,u}=\alpha_{i\,j} \text{ for }v\in C_j\text{\; and \;} \sum\limits_{u\in C_j} A_{u\,v}=\sum\limits_{u\in C_j} A_{v\,u}=\alpha_{i\,j}\text{ for }v\in C_i.$$
\item For any odd $i,j\in[2m]$, $\alpha_{i\,j}=\alpha_{i+1\,j+1}$ and $\alpha_{i+1\,j}=\alpha_{i\,j+1}$.
\end{enumerate}

Let $H$ be the hypergraph which is constructed from $G$ as follows. For all odd $i\in[2m]$, if  $\Gamma(D)\cap(C_i\cup C_{i+1})=C_i$ then remove all edges of the form $\{D,v\}$ with $v\in C_i$ and add the edges of the form $\{D,v\}$ for all $v\in C_{i+1}$.

Then $H$ and $G$ are cospectral with respect to their matrix representation.
\end{theorem}

Observe that by the condition (4) of Theorem~\ref{t.matrix}, we obtain that for all $i,j \in [2m]$, we have $\alpha_{i,j} = \alpha_{j,i}$.

\begin{proof} Let $X=I_t-\frac1t J_t$ and $Y=\frac1t J_t$ be $t\times t$ matrices. Without loss of generalization, we assume that the labelling of vertices of $G$ is consistent with the partition $C_1\cup C_2\cup \dots\cup C_{2m-1}\cup C_{2m}\cup D$.

Consider the block matrix $$Q=\left(\begin{matrix}
X & Y & O & O  & \dots & O \\
Y & X & O & O  & \dots & O \\
O & O & X & Y  & \dots & O \\
O & O &Y & X  & \dots & \vdots \\
\vdots & \cdots & \cdots & \cdots & \cdots & O\\
O & \cdots & \cdots & \cdots & O & I_{k-1}
\end{matrix}\right).$$
It is clear that $Q$ is an orthogonal matrix, which means that $A':=Q A Q^\top$ is a matrix with the same spectrum as $A$.  Hence all we need to prove is that $A'$ is the adjacency matrix of $H$.

We can write $A$ as a block matrix $$A=\left(\begin{matrix}
B_{1\,1} & B_{1\,2} & \cdots & B_{1\,2m} & B_1 \\
B_{2\,1} & B_{2\,2} & \cdots & B_{2\,2m} & B_2 \\
\vdots & \cdots & \cdots & \vdots & \vdots \\
B_{2m\,1} & B_{2m\,2} & \cdots & B_{2m\,2m} & B_{2m} \\
B_1^\top & B_2^\top & \cdots & B_{2m}^\top & t(J-I)
\end{matrix}\right),$$
where $B_{i\,j}$ are $t\times t$ blocks for all $i,j\in[2m]$, while $B_i$ are $t\times (k-1)$ blocks of all-one or all-zero rows. According to the conditions of the theorem satisfied by $G$,
the sum of rows and columns in $B_{i\,j}$ is $\alpha_{i\,j}$ for any $i,j\in[2m]$. The matrix $A'$ can also be represented as a block matrix in a similar way.

Let $B'_{i\,j}$, $B'_i$ for $i,j\in[2m]$ be the $t\times t$ and $t\times (k-1)$ blocks of $A'$, respectively. Let $i$ and $j$ be odd integers from $[2m]$. Then by block matrix multiplication, we can derive that
{\footnotesize{
\begin{align*}
&\left(\begin{matrix}
B'_{i\,j} & B'_{i\, j+1} \\
B'_{i+1\, j} & B'_{i+1\,j+1}
\end{matrix}\right) =\\
&=
\left(\begin{matrix}
(XB_{i\,j}+YB_{i+1\,j})X+(XB_{i\,j+1}+YB_{i+1\,j+1})Y & (XB_{i\,j}+YB_{i+1\,j})Y+(XB_{i\,j+1}+YB_{i+1\,j+1})X  \\
(YB_{i\,j}+XB_{i+1\,j})X+(YB_{i\,j+1}+XB_{i+1\,j+1})Y & (YB_{i\,j}+XB_{i+1\,j})Y+(YB_{i\,j+1}+XB_{i+1\,j+1})X 
\end{matrix}\right).
\end{align*}
}}
Now, since the blocks $B_{i\,j}$ all have constant sum of rows and columns, we have
$$J_tB_{i\,j}=B_{i\,j}J_t=\alpha_{i\,j} J_t\implies YB_{i\,j}=B_{i\,j}Y=\frac{\alpha_{i\,j}}{t}J_t,\; XB_{i\,j}=B_{i\,j}X=B_{i\,j}-\frac{\alpha_{i\,j}}{t}J_t.$$

Additionally, observe that $$X^2=X,\; Y^2=Y,\; XY=YX=O.$$

From this, we can obtain through a straightforward calculation that, for any odd $i,j \in [2m]$, we have
\begin{alignat*}{2}
B'_{i\,j} &= B_{i\,j} - \frac{\alpha_{i,j}}t J_t + \frac{\alpha_{i+1,j+1}}t J_t &\;=\;& B_{i\,j},\\
B'_{i+1\,j} &= B_{i+1\,j} - \frac{\alpha_{i,j+1}}t J_t + \frac{\alpha_{i+1,j}}t J_t  &\;=\;& B_{i+1\,j},\\
B'_{i\,j+1} &= B_{i\,j+1} - \frac{\alpha_{i+1,j}}t J_t + \frac{\alpha_{i,j+1}}t J_t  &\;=\;& B_{i\,j+1},\\
B'_{i+1\,j+1} &= B_{i+1\,j+1} - \frac{\alpha_{i+1,j+1}}t J_t + \frac{\alpha_{i,j}}t J_t  &\;=\;& B_{i+1\,j+1}.
\end{alignat*}
Next, for an odd $i\in[2m]$ we have $B_i'=XB_i+YB_{i+1}$ and $B_{i+1}'=YB_i+XB_{i+1}$. There are two possible cases.

The first case is when $B_i$ is an all-ones block while $B_{i+1}=O$, so that $\Gamma(D)\cap(C_i\cup C_{i+1})=C_i$. Then, since $J_tB_i=tB_i$, we have $B'_i=O$ and $B'_{i+1}=B_i$, meaning that the adjacency of all $k$-sets of the form $\{v,D\}$ for $v\in C_i\cup C_{i+1}$ was switched. This is consistent with the switching used to construct $H$.

The second case is when $B_i$ and $B_{i+1}$ both have exactly $r$ all-ones rows, meaning that $|\Gamma(D)\cap C_i|=|\Gamma(D)\cap C_{i+1}|=r$. Then $J_tB_i=J_tB_{i+1}=rB$ where $B$ is an all-ones block of size $t\times (k-1)$, and from this we can derive $B'_i=B_i$ and $B'_{i+1}=B_{i+1}$.

Combining the observations above we can conclude that $A'$ is the adjacency matrix of $H$.
\end{proof}

Note that the $\frac1{k-1}$ factor in the definition of the adjacency matrix $A$ has no bearing on the argument of the proof, which implies that the above switching is also applicable when using the adjacency matrix definition from \cite{FL1996}.

\begin{remark}
Let $G$ be a hypergraph that satisfies the conditions of Theorem~\ref{t.matrix} with a partition of vertices $C_1\cup C_2\cup D$ and $\Gamma(D)\cap(C_1\cup C_2)=C_1$. Then the application of the switching described in Theorem~\ref{t.matrix} is equivalent to applying the switching described by Sarkar and Banerjee in~\cite{SB2020}. The partition of the vertices in that case would be into two subsets $C:=C_1\cup C_2$ and $D$.
\end{remark}

We end up with an example that illustrates the use of Theorem \ref{t.matrix}. Note that the following cospectral pair of uniform hypergraphs cannot be obtained using the switching for hypergraphs using matrix representation from~\cite{SB2020}.

\begin{example}
Let $G$ be a $3$-uniform hypergraph on 
$14$ vertices $u_1,u_2,\dots,u_{12},v_1,v_2$ and having edge set 
\begin{gather*}
  \{u_1,u_2,u_3\},\; \{u_1,u_4,u_5\},\; \{u_2,u_5,u_6\},\; \{u_3,u_4,u_6\},\; \{u_7,u_{10},u_{12}\},\; \{u_8,u_{10},u_{11}\},\\
  \{u_9,u_{11},u_{12}\},\;\{u_1,v_1,v_2\},\; \{u_2,v_1,v_2\},\; \{u_3,v_1,v_2\},\; \{u_7,v_1,v_2\},\; \{u_{10},v_1,v_2\}.
\end{gather*}

Consider the vertex partition $C_1:=\{u_1,u_2,u_3\}$, $C_2:=\{u_4,u_5,u_6\}$, $C_3:=\{u_7,u_8,u_9\}$, $C_4:=\{u_{10},u_{11},u_{12}\}$, and $D:=\{v_1,v_2\}$. Then the switching of Theorem~\ref{t.matrix} can be applied to obtain a hypergraph $H$ with edge set \begin{gather*}
  \{u_1,u_2,u_3\},\; \{u_1,u_4,u_5\},\; \{u_2,u_5,u_6\},\; \{u_3,u_4,u_6\},\; \{u_7,u_{10},u_{12}\},\; \{u_8,u_{10},u_{11}\},\\
  \{u_9,u_{11},u_{12}\},\;\{u_4,v_1,v_2\},\; \{u_5,v_1,v_2\},\; \{u_6,v_1,v_2\},\; \{u_7,v_1,v_2\},\; \{u_{10},v_1,v_2\}.
\end{gather*}
The constructed hypergraph $H$ has a subset of vertices $\{u_1,u_2,u_3\}$ of degree $2$ which themselves form an edge. There is no such subset in $G$, hence $G$ and $H$ are not isomorphic. 
Furthermore, there is no partition of the vertices of $G$ that satisfies the conditions of the GM-switching based method described in~\cite{SB2020}, meaning that this example can only be obtained using the new method from Theorem \ref{t.matrix}.
\end{example}

\subsection*{Acknowledgements}

We thank Utku Okur and Joshua Cooper for carefully reading the manuscript and pointing out the distinction between cospectrality and $E$-cospectrality of tensors. We also thank the anonymous referees for their useful feedback. Aida Abiad is partially supported by FWO (Research Foundation Flanders) via the grant 1285921N. This research is supported by NWO (Dutch Research Council) via an ENW-KLEIN-1 project (OCENW.KLEIN.475).


\end{document}